\documentclass{amsart}
\usepackage{a4wide}

\usepackage{textcmds} % amsrefs needs this, but it has to be loaded early to avoid re-defs.
\usepackage{amsmath, amssymb, amsfonts, amstext, amsthm, amscd, mathrsfs, mathscinet}
\usepackage{mathtools}
\usepackage{verbatim}
\usepackage{graphicx}
\usepackage{color}
\usepackage{pinlabel}
\usepackage{mathrsfs} 
\usepackage{xypic}
\usepackage{enumitem}
\usepackage{tikz}
\usetikzlibrary{matrix}

\usepackage{etoolbox}
\makeatletter
\let\ams@starttoc\@starttoc
\makeatother
\usepackage[parfill]{parskip}
\makeatletter
\let\@starttoc\ams@starttoc
\patchcmd{\@starttoc}{\makeatletter}{\makeatletter\parskip\z@}{}{}
\makeatother

\usepackage{hyperref}

\newcommand{\C}{\mathbb{C}}
\newcommand{\Z}{\mathbb{Z}}
\newcommand{\R}{\mathbb{R}}

\newcommand{\CP}{\mathbb{C}P}

\newcommand{\OP}{\operatorname}

\renewcommand{\Re}[1]{\mathfrak{Re}\,#1}

\renewcommand{\Re}{\mathfrak{Re}}

\newsavebox{\textvisiblespacebox}
\savebox{\textvisiblespacebox}{\texttt{aa}}
\newcommand\vartextvisiblespace[1][\wd\textvisiblespacebox]{%
 \makebox[#1]{\kern.1em\rule{.4pt}{.3ex}%
 \hrulefill%
 \rule{.4pt}{.3ex}\kern.1em}%
}
\numberwithin{equation}{section}

\newtheorem{thm}{Theorem}[section]
\newtheorem{lma}[thm]{Lemma}
\newtheorem{prp}[thm]{Proposition}
\newtheorem{cor}[thm]{Corollary}

\newtheoremstyle{TheoremNum}
 {\topsep}{\topsep} %%% space between body and thm
 {\itshape} %%% Thm body font
 {} %%% Indent amount (empty = no indent)
 {\bfseries} %%% Thm head font
 {.} %%% Punctuation after thm head
 { } %%% Space after thm head
 {\thmname{#1}\thmnote{ \bfseries #3}}%%% Thm head spec
\theoremstyle{TheoremNum}

\theoremstyle{definition}

\newtheorem{quest}[thm]{Question}

\theoremstyle{remark}
\newtheorem{rmk}[thm]{Remark}

\begingroup 
\makeatletter 
\@for\theoremstyle:=definition,remark,plain,TheoremNum\do{% 
\expandafter\g@addto@macro\csname th@\theoremstyle\endcsname{% 
\addtolength\thm@preskip\parskip 
}% 
} 
\endgroup 

\title{Lagrangian Approximation of Totally Real Concordances}


\author{Georgios Dimitroglou Rizell}
\address{Department of Mathematics\\
Uppsala University\\
Box 480\\
SE-751 06 UPPSALA\\
SWEDEN}
\email{georgios.dimitroglou@math.uu.se}
\thanks{The author is supported by the Knut and Alice Wallenberg Foundation under the grants KAW 2021.0191 and KAW 2021.0300, and by the Swedish Research Council under the grant number 2022-06593 for the Centre of Excellence for Geometry and Physics at Uppsala University.}

\begin{document}

\begin{abstract}
We show that a two-dimensional totally real concordance can be approximated by a Lagrangian concordance whose Legendrian boundary has been stabilised both positively and negatively sufficiently many times. The main applications that we provide are constructions of knotted Lagrangian concordances in arbitrary four-dimensional symplectisations, as well as of knotted Lagrangian tori in symplectisations of overtwisted contact three-manifolds.
\end{abstract}

\maketitle
\setcounter{tocdepth}{1}
\tableofcontents

\section{Introduction and results}

Under most circumstances, Lagrangian submanifolds of symplectic manifolds are rigid object, and a large part of modern symplectic topology is devoted to understanding their subtle obstructions; the most powerful ones known are based upon Gromov's theory of pseudoholomorphic curves and, in addition, the more recently developed microlocal sheaf theory. Notwithstanding, there are several settings in which $h$-principles have been established for Lagrangians, e.g.~open Lagrangian submanifolds \cite{Eliashberg:IntroductionH}, or Lagrangians in high-dimensional symplectic manifolds with overtwisted concave boundary \cite{LagCaps}. We begin by recalling some results that are relevant for our investigations.

First, the $h$-principle for closed Lagrangian embeddings fails in the standard symplectic vector space $(\C^2,\omega_0)$. A manifestation of this is the fact that, although there are knotted totally real tori in $\C^2$ (i.e.~formally Lagrangian tori), all Lagrangian tori are smoothly isotopic through Lagrangians; see work of the author joint with Goodman and Ivrii \cite{Dimitroglou:Isotopy}. In high dimension, the situation is more complicated for tori; see results by the author and Evans \cite{Dimitroglou:Unlinking}. However, for Lagrangians $S^1 \times S^{n-1}$ there are again restrictions; see recent work by Nemirovski \cite{Nemirovski}.

Second, even though an $h$-principle holds for open Lagrangians, this is no longer the case for Lagrangian embeddings with fixed Legendrian boundary in a hypersurface of contact type. For instance, Eliashberg and Polterovich showed in \cite{Eliashberg:LocalLagrangian} that a Lagrangian disc in $D^4$ that coincides with $\mathfrak{Re}\C^2$ near the boundary $S^3=\partial D^4$ must be unknotted. Similarly, an exact Lagrangian cylinder in
$$([\log{\epsilon},0] \times S^3,d(e^t\alpha_{std})) \cong (D^4 \setminus B^4_\epsilon,\omega_0)$$
that coincides with $\mathfrak{Re}\C^2$ near the boundary must be smoothly unknotted by a result by Chantraine, Ghiggini, Golovko and the author \cite[Theorem 4.4]{FloerConc}. Also see Theorem \ref{thm:unknotted} below for a generalisation.

On the other hand, exact Lagrangian embeddings inside a Liouville cobordism $(\overline{X},d\lambda)$ with non-empty concave boundary $\partial_- \overline{X}$ are rather flexible under certain assumptions on $\partial_-\overline{X}$. More precisely, when $\dim \overline{X}=4$ it was shown by Lin in \cite{Lin} that a sufficiently stabilised knot inside $\partial_-\overline{X}$ admits an exact Lagrangian cap, i.e.~an exact Lagrangian with boundary coinciding with the given stabilised Legendrian. When $\dim \overline{X} \ge 6$ Eliashberg and Murphy introduced an $h$-principle for Lagrangians whose boundary in $\partial_- \overline{X}$ is stabilised. Recall that, in contact manifolds of dimension at least $\dim Y^{2n+1} \ge 5$, a stabilised Legendrian is loose in the sense of Murphy \cite{LooseLeg}, which is a class of Legendrian embeddings that satisfies an $h$-principle. Note that there is a large class of Legendrians that are not loose, for which no analogous $h$-principle is valid.

In contact manifolds of dimension three there is no class of Legendrians analogous to the class of loose Legendrians in high dimensions. However, adding sufficiently many stabilisations to a Legendrian, tends to make it more and more flexible. The Legendrian stabilisation in this dimension is the local modification of adding a zig-zag in the front projection as depicted in Figure \ref{fig:winding}; the stabilisation has a sign that depends on the orientation. For instance, Fuchs and Tabachnikov have shown that two smoothly isotopic Legendrian knots become Legendrian isotopic after sufficiently many stabilisations of both signs \cite{FuchsTabachnikov}. Also, it is easy to see that any smooth knot can be $C^0$-approximated by a Legendrian knot in the same smooth isotopy class if one allows sufficiently many stabilisations of both signs; see e.g.~work by Etnyre~\cite{Etnyre:Legendrian}, work by Cahn and Chernov \cite{Cahn}, as well as Proposition \ref{prp:nowherereeb} in this article. The reason why high dimensions is so different is that the number of stabilisations is not an isotopy invariant; roughly speaking, after stabilising once, any additional stabilisation does not alter the Legendrian isotopy class.

In this work we are only concerned with the low-dimensional case. Indeed, our flexibility result also holds only after one has performed sufficiently many stabilisations of \emph{both signs}.

Our main object of focus is that of a two-dimensional \textbf{Lagrangian concordance} in the four-dimensional symplectisation $(\R_t \times Y^3,d(e^t\alpha))$ from a Legendrian $\Lambda_-$ to a Legendrian $\Lambda_+$. By this, we mean a proper embedding of a Lagrangian cylinder $\R \times S^1 \hookrightarrow \R \times Y$ which coincides with
$$(-\infty,-1]\times \Lambda_- \:\: \cup \:\: [1,+\infty) \times \Lambda_+ \:\: \subset \:\: (\R _t \times Y^3,d(e^t\alpha))$$
outside of a compact subset. We will also consider compact Lagrangian concordances, which are Lagrangian embeddings
$$([T_-,T_+] \times S^1,\{T_-\} \times S^1,\{T_+\} \times S^1) \hookrightarrow ([T_-,T_+] \times Y,\{T_-\} \times Y, \{T_+\} \times Y)$$
of a compact cylinder which is tangent to $\partial_t$ near its boundary. We say that such a concordance $\Sigma \subset [T_-,T_+]\times Y$ has concave boundary $\Lambda_-$ and convex boundary $\Lambda_+$ if $\Sigma \cap\{t=T_\pm\}=\{T_\pm\} \times \Lambda_\pm$.

A symplectisation $(\R \times Y,d(e^t\alpha))$ has a natural set of compatible almost complex structures that are \textbf{cylindrical}, i.e.
\begin{itemize}
\item $J\xi=\xi$ is a compatible almost complex structure on $\xi=\ker\alpha \cap \ker dt$;
\item $J$ is invariant under translation of the $\R$-factor;
\item $J\partial_t=R_\alpha$ is the Reeb vector field defined by $\alpha(R_\alpha)=1$, $d\alpha(R_\alpha,\cdot)=0$, and $dt(R_\alpha)=0$.
\end{itemize}

In the case $Y=S^3$ or $\R^3$, the symplectisation $\R \times Y$ carries a natural cylindrical integrable complex structure; in the case $Y=S^3$, this is the complex structure induced by the symplectomorphism
\begin{gather*}
(\R \times S^3,d(e^t\alpha_{st})) \cong (\C^2 \setminus \{0\},\omega_0),\\
(t,p) \mapsto e^{t/2}\cdot p,
\end{gather*}
while in the case $Y=\R^3$ it is induced by the symplectomorphism
\begin{gather*}
\left(\R \times \R^3,d\left(e^t\left(dz-y\,dx\right)\right)\right) \cong (\{\C^2; \:\Re(z_2) >0\},e^{x_1+y_2^2/2}\left(dy_1-y_2\,dx_2\right)),\\
(t,x,y,z) \mapsto (t-y^2/2+iz,x+iy)
\end{gather*}

Consider a two-dimensional concordance $\Sigma \subset ([T_-,T_+] \times Y^3,d(e^t\alpha),J)$ in a four-dimensional symplectisation as above, where $J$ is a compatible almost complex structure, but where instead of the Lagrangian condition we merely assume that $\Sigma$ is totally real in the sense that $T\Sigma \cap JT\Sigma=\{0\}$ is transverse. Furthermore, we assume that $J$ is cylindrical near the boundary $\{T_\pm\} \times Y$ and, again, that $\Sigma$ is tangent to $\partial_t$ there. In other words, being totally real near the boundary is equivalent to $K_\pm=\Sigma\cap (\{T_\pm\} \times Y)$ being nowhere tangent to the Reeb vector field of $\alpha$.

\begin{thm}
\label{thm:main}
After choosing $T_+' \gg T_+$ sufficiently large, and extending $\Sigma$ by adjoining the trivial cylinder $[T_+,T_+'] \times K_+$, there exists a $C^0$-small smooth isotopy of $([T_-,T'_+] \times Y^3, \{T_-,T'_+\} \times Y^3)$ that takes the concordance $\Sigma$ to a Lagrangian concordance $L$ with convex and concave boundary given by Legendrian $C^0$-approximations $\Lambda_\pm$ of $K_\pm$. Moreover:
\begin{itemize}
\item If we assume that $\R \times A \subset \Sigma$ holds for some closed arc $A \subset K_\pm$ that is Legendrian, we may assume that the deformation and isotopy of $\Sigma$ is fixed inside $\R \times B$ for any open $B \subset Y$ that satisfies $A \cap B \subset \OP{int} A$.
\item If we assume that the boundaries $K_\pm$ of $\Sigma$ are Legendrian, then both $\Lambda_\pm$ can be constructed from $K_\pm$ by adding $k$ positive and $k$ negative stabilisations, where $k \gg 0$. When, in addition, the previous bullet point is satisfied for some arc $A \subset K_\pm$, we can again assume that the deformation and isotopy of $\Sigma$ is fixed inside a subset $\R \times B$ as above.
\end{itemize}
\end{thm}
Recall that totally real submanifolds (closed, open, or with boundary conditions) satisfy all forms of $h$-principles; see \cite[27.4.1]{Eliashberg:IntroductionH}. This flexibility implies the existence of plenty of totally real concordances in different smooth isotopy classes. The main point with our result is that this flexibility also gives rise to a wealth of Lagrangian concordances in different smooth isotopy classes, under the assumption that the Legendrian boundary is stabilised sufficiently both positively and negatively. 
\begin{rmk}
Instead of specifying a compatible almost complex structure and working with totally real concordances, one could alternatively consider the class of $\epsilon$-Lagrangian embeddings which also satisfy an analogous $h$-principle; see \cite[27.2]{Eliashberg:IntroductionH}.
\end{rmk}

Next we present our two main applications. The first one shows that the smooth knotted class of Lagrangian concordances becomes more and more flexible after we add sufficiently many stabilisations. This is in contrast to the unknottedness results for Lagrangian concordances between the standard (unstabilised) unknot in $S^3$ and itself that was proven in \cite{FloerConc}; also see Theorem \ref{thm:unknotted} below.

\begin{thm}
\label{thm:knotted}
Consider a Legendrian knot $\Lambda \subset (S^3,\xi)$ inside a contact sphere.
\begin{itemize}
\item \emph{When $\xi=\xi_{st}$ is the standard tight contact structure:} there exist numbers $k_\pm \ge 0$ depending on $N \ge 0$ with the property that, after constructing a Legendrian $\Lambda'$ obtained by adding $k_+$ positive and $k_-$ negative stabilisations to $\Lambda$, there exists $N$ pairwise different Lagrangian concordances from $\Lambda$ to itself up to homeomorphism of $\R \times S^3$.
\item \emph{When $\xi=\xi_{ot}$ is overtwisted:} There exists infinitely many Lagrangian concordances $C$ from $\Lambda$ to itself up to homeomorphism of $\R \times S^3$ whenever $\Lambda$ admits an overtwisted disc in its complement.
\end{itemize}
\end{thm}
\begin{quest}
How many stabilisations $k_+,k_-\ge0$ are needed in the first bullet point for some given knot class? Is it universally bounded or not?
\end{quest}

We use the above flexibility to refine the existence of exact Lagrangian tori in symplectisations of overtwisted contact manifolds proven by the author in \cite{Dimitroglou:Exact} to show that the produced tori also can be taken to be smoothly knotted.

\begin{thm}
\label{thm:ot}
There exist infinitely many smooth isotopy classes of exact Lagrangian tori with vanishing Maslov class inside the symplectisation $(\R \times S^3,d(e^t\alpha_{ot}))$ when $\ker \alpha_{ot}$ is an overtwisted contact structure on $S^3$.
\end{thm}

The first construction of exact Lagrangian tori in these symplectisation was carried out in \cite{Dimitroglou:Exact} by the author, where it was shown that the Maslov class moreover can be taken to vanish. The outcome of the present work is that they also can be taken to be smoothly knotted. Recall that there are no exact Lagrangian tori in the symplectisation of the standard tight $S^3$, and that all Lagrangian tori in that symplectic manifold are smoothly isotopic by \cite{Dimitroglou:Isotopy}.

Finally we prove a rigidity result that exhibits the necessity to use stabilisations of \emph{both} signs in order to guarantee the existence of knotted concordances.
\begin{thm}
\label{thm:unknotted}
If $\Lambda_k$ is a $k$-fold stabilisation of the standard Legendrian unknot of ${\tt tb}=-1$, where all stabilisations have the same sign, then any Lagrangian concordance $L \subset \R \times S^3$ from $\Lambda_k$ to itself is compactly supported smoothly isotopic to $\R \times \Lambda_k$.
\end{thm}
\begin{quest}
Are any two such Lagrangian concordances for the same value $k$ necessarily Hamiltonian isotopic through Lagrangian concordances? This is the case when $k=0$ by \cite[Theorem 4.4]{FloerConc}.
\end{quest}
 The proof of the above result uses techniques similar to the ones used by Eliashberg and Polterovich in \cite{Eliashberg:ProblemLagrangian}, where smooth isotopy classes of Lagrangian cobordisms were classified by, first, deforming them to symplectic cobordisms, and then making them $J$-holomorphic and using the technique of persistence of pseudoholomorphic foliations that is available in dimension four.

\section{Smooth unknottedness (Proof of Theorem \ref{thm:unknotted})}

We choose the unique orientation of the unknot $\Lambda_k$ that makes all stabilisations negative. (This is possible by the assumption that all stabilisations are of the same sign.) Arguing as in \cite[Lemma 4.1]{Cao} we can push off the Lagrangian concordance $L$ to a symplectic concordance $\Sigma$ that coincides with the cylinder over the positive transverse push-off of $T_+(\Lambda_k)$ outside of a compact subset, where the positivity of the transverse push-off is with respect to the choice of orientation that was made before. Roughly speaking, this push-off is constructed in a standard Weinstein neighbourhood $D^*L$ of the cobordism, where it is given by a section that is of the one-form form $e^td\theta$. Here we must use a suitable identification of $L=\R_t \times S^1_\theta$ and a suitable Weinstein neighbourhood that is compatible with the cylindrical structure of $\R \times S^3$.

We will show that $\Sigma$ is compactly supported Hamiltonian isotopic to such a symplectic push-off of the trivial concordance $\R \times \Lambda_k$, which implies the claim.

It is a standard fact that isotopies of transverse knots in contact manifolds are generated by global contact isotopies $\phi_t \colon (Y,\alpha) \to (Y,\alpha)$, $\phi_t^*\alpha=e^{g_t}\alpha$. Hence, the isotopy of symplectic cylinders $\R \times T_t \subset \R \times Y$ over a smooth family $T_t=\phi_t(T_0)$ of transverse knots can be generated by a global cylindrical Hamiltonian isotopy
\begin{gather*}
(\R_t \times Y,d(e^t\alpha)) \to (\R_t \times Y,d(e^t\alpha)),\\
 (t,y) \mapsto (t-g_t(y),\phi_t(y)),
 \end{gather*}
which, moreover, preserves the primitive $e^t\alpha$ of the symplectic form. Recall that the transverse isotopy class of the positive transverse push-off of a Legendrian is preserved under negative stabilisations of the Legendrian; see e.g.~\cite[Section 2.9]{Etnyre:Legendrian} for definitions of the transverse push-off and its behaviour under stabilisation. It follows that the above symplectic push-off $\Sigma$ can be identified with a symplectic concordance that is cylindrical over the positive transverse push-off $T_+(\Lambda_0)$ of the standard Legendrian unknot at both ends, under a cylindrical Hamiltonian diffeomorphism. 

Furthermore, since the transverse unknot $T_+(\Lambda_0)$ is transverse isotopic to the standard periodic Reeb orbit $S^1 \times \{0\} \subset S^3$ we have concluded the following: there exists a cylindrical Hamiltonian diffeomorphism of $\R \times Y$ that takes $\Sigma \subset \R \times S^3$ to a symplectic concordance that coincides with $\R \times (S^1 \times \{0\}) \subset S^3$ outside of a compact subset.

Since cylindrical Hamiltonian diffeomorphisms preserve trivial cylinders $\R \times K \subset \R \times S^3$, it is now sufficient to show that any symplectic concordance $\Sigma$ that coincides with $\R \times (S^1 \times \{0\}) \subset S^3$ outside of a compact subset, is compactly supported isotopic to $ \R \times (S^1 \times \{0\})$.

There is a symplectomorphism 
\begin{gather*}
\left(\left(-\infty,e^a\right) \times S^3,e^t\alpha\right) \to \left(B^4 \setminus \{0\},e^a\omega_0\right)\\
(t,p) \mapsto e^{t/2-a}\cdot p
\end{gather*}
where the target space can be identified with $\left(\CP^2\setminus \left(\CP^1_\infty \cup \{0\}\right),e^a\omega_{FS}\right)$; here we have normalised the Fubini--Study form $\omega_{FS}$ so that the integral over a line is $\int_{\CP^1_\infty} \omega_{FS}=\pi$, where $\CP^1_\infty$ denotes the line at infinity. This symplectomorphism can be assumed to map $\Sigma$ to a symplectic surface that coincides with the complex line
$$\overline{\C \times \{0\}} \setminus \left(\{0\} \sqcup \CP^1_\infty\right) \subset \CP^2 \setminus \left(\{0\} \sqcup \CP^1_\infty\right) $$
outside of a compact subset.

We proceed by making $\Sigma$ into a $J$-holomorphic concordance for a compatible almost complex structure $J$ on $(\CP^2,e^a\omega_{FS})$, where $J$ is taken to coincide with the standard complex structure $J_0$ near $\{0\} \cup \CP^1_\infty \subset \CP^2$. In particular, the compactification $\overline{\Sigma} \subset \CP^2$ is a symplectic sphere that coincides with $\overline{\C \times \{0\}}$ near $\{0\} \cup \CP^1_\infty$.

Then we proceed with Gromov's argument from \cite[2.4]{Gromov:Pseudo}. Choose a one-parameter family $J_t$ of compatible almost complex structures where $J_1=J$ and $J_0=i$ is the standard complex structure, and where all $J_t$ coincide with $i$ in a neighbourhood of $\{0\} \sqcup \CP^1_\infty$. The $J_t$-holomorphic sphere $\ell_t$ of degree one tangent to $\C \times \{0\}$ at $\{0\}$ is uniquely determined by this requirement, and vary smoothly with $t$. (Since $J_t=J_0$ near $\{0\}$ the tangency is indeed $J_t$-complex for all $t$.) In particular, $\ell_t$ is a smooth isotopy of symplectic spheres that satisfy $\ell_1=\overline{\Sigma}$ and $\ell_0=\overline{\C \times \{0\}}$.

Note that the smooth family $\ell_t$ of $J_t$-holomorphic lines are not necessarily cylindrical outside of a compact subset of $(-\infty,e^a) \times S^3$, except in the cases $t \in \{0,1\}$. What remains is to smoothly deform the family of lines $\ell_t$ for all $t \neq \{0,1\}$, while fixing $\ell_0$ and $\ell_1$, in order to make $\ell_t \setminus (\{0\} \cup \CP^1_\infty)$ stay cylindrical for all $t$. We finish the proof by giving this argument.

As in the proof of \cite[Corollary 3.7]{Dimitroglou:Isotopy} we can find a Hamiltonian isotopy that, first, makes $\ell_t$ all intersect $\CP^1_\infty$ in the fixed point $\overline{\C \times \{0\}} \cap \CP^1_\infty$ (this can be done by a Hamiltonian isotopy of $\CP^1_\infty$) and, second, makes it tangent to $\overline{\C \times \{0\}}$ there (this straightens the curve at the intersection). This Hamiltonian isotopy can be taken to be supported in a small neighbourhood of $\CP^1_\infty$. Since the family of spheres now is tangent to $\overline{\C \times \{0\}}$ at both points
$$\{0\} \:\:\text{and} \:\: \overline{\C \times \{0\}} \cap \CP^1_\infty \subset \CP^2,$$
one can now readily find a Hamiltonian isotopy that ``straightens'' the spheres to make them coincide with $\overline{\C \times \{0\}}$ in some small neighbourhood of these tangencies.

This deformed family of lines $\tilde{\ell}_t$ finally gives the sought isotopy from $\Sigma$ to $\R \times (S^1 \times \{0\})$ with compact support.
\qed

\section{Proof of the applications (Theorems \ref{thm:knotted} and \ref{thm:ot})}

\subsection{Fundamental group computations}

\label{sec:fundgroup}

In order to distinguish the isotopy classes of the constructed surfaces, we need to perform computations of the fundamental group of the surface complements. Here we provide the necessary tools.

The following lemma computes the fundamental group of the complement of a surface that is ``unknotted,'' i.e.~a surface that bounds an embedded three-dimensional handle-body. This computation is standard; see e.g.~\cite[Section 1.2]{Kamada} for the case of surfaces of genus zero and one.
\begin{lma}
A closed oriented surface $\Sigma \subset \R^4$ that bounds a handle-body satisfies
$$ \pi_1 (\R^4 \setminus \Sigma)=\langle\mu\rangle =\Z\mu,$$
where $\mu$ is a meridian in the fibre of a small spherical normal bundle.
\end{lma}
\begin{proof}
Let $H \subset \R^4$ be an embedded handle-body with boundary $\partial H=\Sigma$ that exists by assumption.

The existence of the handle-body implies that $\Sigma$ has a trivial normal bundle. Using $U_\Sigma$ to denote a small tubular neighbourhood of $\Sigma$, the triviality of the normal bundle gives a homotopy equivalence $U_\Sigma \setminus \Sigma \sim \Sigma \times S^1$ under which the intersection $H \cap U_\Sigma$ becomes a section $\Sigma \times \{1\}$. Consider a tubular neighbourhood $U_H \supset H$ that consists of $U_\Sigma$ together with a small normal bundle of $H$ along the interior of $H$. The neighbourhood $U_H \setminus \Sigma$ has the homotopy type of the glued space
$$(\Sigma \times S^1 \: \sqcup \: H)/(\partial H \sim \Sigma \times \{1\}).$$
Using Seifert--van Kampen one readily computes
$$ \pi_1(U_H \setminus \Sigma)=\pi_1(H) \times \Z\mu, $$
where $\mu$ is the generator of the meridian of the fibre of the normal bundle $U_\Sigma \to \Sigma$.

The handle-body $H$ has the homotopy type of a wedge of $1$-spheres. This is also true for the neighbourhood $U_H \supset H$ of which $H$ is a deformation retract. In particular, as follows by a genericity argument, the inclusions $ \R^4 \setminus H \subset \R^4$ and $U_H \setminus H \subset U_H$ induce isomorphisms
$$\pi_1(\R^4 \setminus H) \to \pi_1(\R^4) \cong\{1\} \:\:\text{and}\:\: \pi_1(U_H \setminus H) \to \pi_1(U_H) \cong \pi_1(H)$$
of fundamental groups.

Another application of Seifert--van Kampen now shows that 
$$ \pi_1(\R^4\setminus \Sigma)=\pi_1(\R^4 \setminus H) \star_{\pi_1(U_H \setminus H)} \pi_1\left(U_H \setminus \Sigma\right)=\{1\} \star_{\pi_1(H)} \left(\pi_1(H) \times \Z\mu\right)=\Z\mu.$$
which is the sought identity.
\end{proof}

\begin{figure}[htp]
 \vspace{3mm}
 \labellist
 	\pinlabel $p_x$ at 202 25
 \pinlabel $p_{\theta}$ at 150 58
 \pinlabel $x$ at 124 -1
 	\pinlabel $p_x$ at 95 25
 \pinlabel $p_{\theta}$ at 42 58
 \pinlabel $x$ at 16 -1
	\endlabellist
 \includegraphics[scale=1.2]{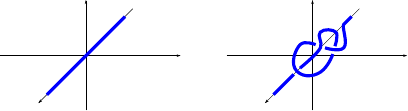}
 \caption{Left: An unknotted arc. Right: A knotted arc that coincides with $\{p_{\theta}=p_x=0\}$ near the boundary and which is nowhere tangent to $\partial_{p_{\theta}}$.}
 \label{fig:knot}
\end{figure}

The main technique that we will apply for constructing knotted surfaces is by excising an annulus from the surface, and then replacing it with a knotted annulus. We proceed with the precise recipe. Consider a surface $(\Sigma,\partial \Sigma) \subset (X^4,\partial X)$ possibly with empty boundary, and a parametrised neighbourhood that intersects the surface in an embedded annulus $\R \times S^1$ 
$$\iota \colon (\R^3 \times S^1, \{(0,0)\} \times \R \times S^1) \hookrightarrow (X\setminus \partial X, \Sigma\setminus \partial X).$$
Let $\Sigma_{\iota,A}$ denote the surface obtained from $\Sigma$ by replacing the annular subset
$\iota(\{(0,0)\} \times \R \times S^1)$ with the knotted annulus
$\iota(A \times S^1)$ for a possibly knotted arc $A \subset \R^3$ that coincides with $\{(0,0)\} \times \R$ outside of $B^3_{R} \subset \R^3$ for some $R > 0$. See Figure \ref{fig:knot} for an example.

Below we investigate how this procedure changes the fundamental group of the complement of the surface, and hence its diffeomorphism class. Since we want to express the change of fundamental group in terms of knot groups (i.e.~the fundamental group of the complement of a compact knot in $\R^3$ or $S^3$), it will be useful to consider the knot $K_A \subset S^3$ obtained by closing up the arc $A \cap B^3_{2R}$ inside $ B^3_{3R} \setminus B^3_R \subset S^3$ by adjoining an unknotted arc. This can e.g. be done by smoothing the piece-wise smooth knot obtained by adjoining an arc entirely contained inside $\partial D^3_{2R}$ that connects the two endpoints of $A \cap \partial D^3_{2R}$. Note that there is a homeomorphism $$\R^3 \setminus A \cong S^3 \setminus K_A$$ and thus $\pi_1(\R^3 \setminus A) =\pi_1(S^3 \setminus K_A)$. We illustrate this construction in Figure \ref{fig:knot2}.

\begin{figure}[htp]
 \vspace{3mm}
 \labellist
 	%\pinlabel $p_x$ at 202 25
 %\pinlabel $p_\varphi$ at 150 58
 %\pinlabel $x$ at 124 -1
 	%\pinlabel $p_x$ at 95 25
 %\pinlabel $p_\varphi$ at 42 58
 %\pinlabel $x$ at 16 -1
 \pinlabel $\color{blue}A$ at 50 15
 \pinlabel $\color{blue}K_A$ at 185 15
	\endlabellist
 \includegraphics[scale=1.2]{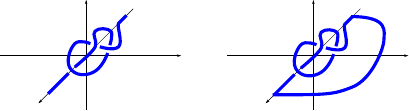}
 \caption{A long properly embedded knotted arc $A \subset \R^3$ and a knot $K_A \subset \R^3$ obtained by closing up the arc; the fundamental groups of the complements in $\R^3$ of these two submanifolds are isomorphic.}
 \label{fig:knot2}
\end{figure}

The following lemma is a direct application of Seifert--van Kampen. It can be seen as a version of the well-known knot-spinning construction, which is a classical method for producing surfaces that are smoothly knotted; see e.g.~\cite[Section 6.1]{Kamada}.
\begin{lma}
\label{lma:svk}
There is an isomorphism
$$
\pi_1\left(X \setminus \Sigma_A\right) \cong 
\pi_1(U_1) \star_{\Z\mu\times\Z\lambda} \pi_1(U_2),
$$
that is induced by the decomposition
\begin{gather*}
X \setminus \Sigma_{\iota,A}= U_1 \cup U_2,\:\:\text{for}\:\: U_1 =X \setminus \left(\Sigma\cup \iota\left( D^3_{2R} \times S^1 \right)\right), \:\: U_2= \iota\left( B^3_{3R} \times S^1\right)\setminus \Sigma_{\iota,A},
\end{gather*}
where the inclusions $U_1 \subset X \setminus \Sigma$ and $(B^3_{3R} \times S^1) \setminus (A \times S^1) \subset (\R^3\setminus A) \times S^1$ induce isomorphisms $\pi_1(U_1)=\pi_1\left(X \setminus \Sigma \right)$ and $\pi_1(U_2)= \pi_1\left(S^3 \setminus K_A\right) \times \Z\lambda,$
and where we have made the canonical identification
$$ \pi_1(U_1 \cap U_2)=\pi_1\left(\left(S^2 \setminus \{\pm N\} \times \R\right) \times S^1\right) = \pi_1\left(S^2 \setminus \{\pm N\} \right) \times \pi_1(S^1) = \Z\mu \times \Z\lambda $$
of the fundamental groups.
\end{lma}

The following two corollaries follow immediately from the above lemma. For the first one, we use the fact that the knot group of the connected sum $K\sharp K_A$ of knots satisfies
$$ \pi_1(K \sharp K_A)=\pi_1(K) \star_{\Z\mu} \pi_1(K_A) $$
where $\Z\mu \subset \pi_1(K), \pi_1(K_A)$ is the subgroup generated by the meridian.

\begin{cor}
\label{cor:cyl}
Assume that $X=\R\times S^3$ and $\Sigma = \R \times K$, which implies that $\pi_1(X \setminus \Sigma) = \pi_1(S^3 \setminus K)$. When the class $\lambda$ from Lemma \ref{lma:svk} above is trivial inside $\pi_{1}(U_1) =\pi_1(X \setminus \Sigma)$, which e.g.~is the case if the annulus in $\Sigma$ parametrised by $\iota|_{\{(0,0)\} \times \R \times S^1}$ is contractible as a map, then there is an isomorphism of fundamental groups
$$\pi_1\left(X \setminus \Sigma_{\iota,A}\right) \cong \pi_1\left(S^3 \setminus \left(K \sharp K_A\right)\right)$$
where $K\sharp K_A$ denotes the connected sum of knots.
\end{cor}

\begin{cor}
\label{cor:tori}
Assume that $X=\R \times S^3$ and that $\Sigma \subset X$ is an embedded unknotted torus (i.e. which bounds an embedded three-dimensional handle-body). When the class $\lambda$ from Lemma \ref{lma:svk} above is trivial inside $\pi_{1}(U_1) =\pi_1(X \setminus \Sigma)$, which e.g.~is the case if the annulus in $\Sigma$ parametrised by $\iota|_{\{(0,0)\} \times \R \times S^1}$ is contractible as a map, then 
there is an isomorphism 
$$\pi_1\left(X \setminus \Sigma_{\iota,A}\right) \cong \pi_1\left(S^3 \setminus K_A\right)$$
of fundamental groups.
\end{cor}

\subsection{Constructing knotted Lagrangian cylinders}

Here we give a general technique for constructing knotted Lagrangians by deforming a cylindrical Lagrangian in an annular subset. We first perform the construction from Subsection \ref{sec:fundgroup}, deforming the Lagrangian surface near an embedded annulus, in order to produce a totally real surface in a different smooth isotopy class. This construction is related to the so-called spinning construction, which is a general method for producing knotted surfaces; see \cite[Section 6.1]{Kamada}. Then we deform the totally real surface to a Lagrangian by applying Theorem \ref{thm:knotted}.

We start by considering a trivial Lagrangian cylinder
$$ \R \times \Lambda \subset (\R_t \times Y,d(e^t\alpha))$$
in the symplectisation. Any smooth embedding $i \colon [-1,1]_x \times S^1_\theta \hookrightarrow \R \times \Lambda$ of an annulus can be extended to a Weinstein neighbourhood. More precisely, there exists a symplectic embedding
$$ \iota \colon \left(D^*\left([-1,1]_x \times S^1_{\theta}\right),d\left(p_x\,dx+p_{\theta}\,d\theta\right)\right) \hookrightarrow \left(\R_t \times Y,d\left(e^t\alpha\right)\right)$$
that restricts to the original embedding $i$ on the zero-section $0_{[-1,1] \times S^1}$. In addition we fix the choice of a compatible almost complex structure $J$ that satisfies the property that $J \partial_{\theta}=-\partial_{p_{\theta}}$.

Replacing the Lagrangian annulus $0_{[-1,1] \times S^1}$ with an annulus of the form
$$ A \times 0_{S^1} \subset D^*([-1,1] \times S^1) $$
where the subset
$$A \subset D^*[-1,1] \times \R_{p_{\theta}}=[-1,1]_x \times \R_{p_x} \times \R_{p_{\theta}}$$
is a knotted arc that coincides with $0_{[-1,1]} \times \{0\}$ near the boundary provides a realisation of the construction of $(\R \times \Lambda)_{\iota,A}$ described above. In the case when $A$ is taken to be nowhere tangent to the last coordinate $\partial_{p_{\theta}}$, which always can be achieved after a generic small perturbation, the new cobordism is moreover totally real for $J$ and, of course, still Lagrangian outside of a compact subset. See Figure \ref{fig:knot} for an example of a knotted arc of the sought form.

The following result is now an immediate corollary of Theorem \ref{thm:main}.
\begin{cor}
\label{cor:lagrel}
Denote by $\Lambda_k$ the Legendrian obtained from $j^10 \subset J^1\R$ by performing $k$ stabilisation of both signs. For any $\iota$ and $A$ as above, for $k \gg 0$ sufficiently large, we can find a Lagrangian concordance $L$ from $\Lambda_k$ to itself which is smoothly isotopic to $(\R \times j^10)_{\iota,A}$ through concordances that are cylindrical outside of a compact subset.

Furthermore, after a rescaling of the contact vector-space $J^1 \times \R = \R^3$, we may assume that $\Lambda_k$ coincides with $j^10$ outside of some arbitrarily small compact neighbourhood $U$, and that $L$ coincides with $\R \times \Lambda_k$ outside of some arbitrarily small compact subset of $\R \times Y$ that is contained inside $\R \times U$.
\end{cor}

\subsection{Proof of Theorem \ref{thm:knotted}}

\emph{The case of the standard contact structure $\xi=\xi_{std}$ on $S^3$:}

We consider an arbitrary Legendrian knot $\Lambda \subset S^3$ and its Lagrangian cylinder $\R \times \Lambda$. Consider a small contact Darboux chart $J^1\R \hookrightarrow S^3$ in which $\Lambda$ is identified with $j^10$. Add $k \gg 0$ many stabilisations of both signs to the latter Legendrian arc inside the Darboux chart, and denote the resulting stabilised Legendrian knot by $\Lambda_k$. Corollary \ref{cor:lagrel} now provides a Lagrangian realisation of $(\R \times \Lambda)_{\iota,A}$ for any knotted arc $A$ and embedding $\iota$ as above, where this Lagrangian concordance coincides with the trivial cylinder $\R \times \Lambda$ outside of the cylinder over the above Darboux chart. This Lagrangian concordance has both ends over $\Lambda_k$.

We may use Corollary \ref{cor:cyl} to compute the fundamental group of the complement of the deformed Lagrangian cobordism $(\R \times \Lambda)_{\iota,A}$. Consider $N$ different smooth knot classes $K_{A_1},\ldots,K_{A_N} \subset S^3$ that have pairwise non-isomorphic knot groups $\pi_1(S^3 \setminus K_{A_i})$. Taking $k \gg 0$ sufficiently large and assuming that we started with an unknotted Legendrian $\Lambda \subset S^3$, we can construct Lagrangian concordances from $\Lambda_k$ to itself whose complements realise all these different knot groups. The result follows from this.

\emph{The case of an overtwisted contact structure $\xi=\xi_{ot}$ on $S^3$:}

Consider Legendrian knot $\Lambda \subset S^3$ that lives in the complement of an overtwisted disc, and its Lagrangian cylinder $\R \times \Lambda$.

Recall the following standard result, see e.g.~\cite[Lemma 2.7]{Dimitroglou:Exact}. 
\begin{lma}
Any Legendrian knot $\Lambda \subset (Y,\xi)$ that lives in the complement of an overtwisted disc can be obtained from some different Legendrian knot in the same smooth isotopy class by adding $k$ stabilisations of both signs.
\end{lma}
\begin{rmk}
The Legendrian $\Lambda$ is contained in the complement of overtwisted discs, but the Darboux balls in which the stabilisations live will typically need to intersect some overtwisted disc.
\end{rmk}
\begin{proof}
This statement follows from the standard fact that the unknot in the complement of an overtwisted disc is stabilised, which e.g.~can be proven using Giroux' theory of convex surfaces \cite{Giroux:Convexite}. Using this fact, one can then perform connected sum of $\Lambda$ with $2k$ suitably oriented unknots. Since connected sums with unknots preserve the Legendrian isotopy class, we conclude that $\Lambda$ admits $k$ positive and $k$ negative destabilisations.

Alternatively, one can use Giroux' theory of convex surfaces directly to exhibit the existence of the stabilisations. Recall that any Legendrian $\Lambda$ can be realised as the boundary of a convex annulus with Legendrian boundary, with has precisely $4k+2$ diving curves, where the boundary of each dividing curve meets both boundary components of the annulus; see the left-hand side of Figure \ref{fig:destab} which shows the case $k=1$. Since this annulus can be assumed to live in the complement of an overtwisted disc, its complement is overtwisted.

The flexibility entailed by the overtwistedness of the complement allows us to extend the aforementioned annulus to a convex annulus as shown on the right-hand side of Figure \ref{fig:destab}, for which one boundary component coincides with a boundary component of the old annulus, while the other boundary component is a Legendrian knot $\Lambda'$ obtained from $\Lambda$ by $k$ positive and $k$ negative de-stabilisations. In other words, $\Lambda$ admits $k$ positive and $k$ negative destabilisations in some Darboux chart. \end{proof}

\begin{figure}[htp]
 \vspace{3mm}
 \labellist
 \pinlabel $\color{blue}\Lambda$ at 10 18
 \pinlabel $\color{blue}\Lambda'$ at 140 15
	\endlabellist
 \includegraphics[scale=1.5]{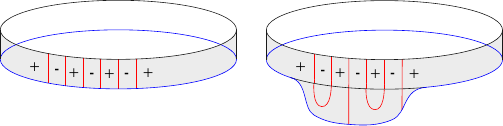}
 \caption{Left: A convex annulus with Legendrian boundary and six dividing curves, each with boundary intersecting both components of the boundary of the annulus. Right: An extension of the annulus along one of the boundary component $\Lambda$ (shown in blue on the left), exhibiting a positive and a negative stabilisation of that Legendrian. The destabilised Legendrian is $\Lambda'$ shown in blue on the right.}
 \label{fig:destab}
\end{figure}

Since any Legendrian that lives in the complement of $\Lambda$ admits arbitrarily many stabilisations of both signs by the above lemma, we can find a standard Darboux neighbourhood $J^1\R$ in which $\Lambda$ is identified with a Legendrian arc $\Lambda_k \subset J^1\R$ obtained from $j^10$ by $k \gg 0$ stabilisations of both signs.

Since $k\gg0$ can be taken to be arbitrary (for a suitable Darboux neighbourhood), Corollary \ref{cor:lagrel} can be used to construct a Lagrangian realisation of $(\R \times \Lambda)_{\iota,A}$ for any knotted arc $A$ and embedding $\iota$ as above. We can then argue as in the case $\xi=\xi_{std}$ above to produce Lagrangian concordances between $\Lambda$ and itself whose complements have fundamental groups that realise infinitely many different knot groups.
\qed

\subsection{Proof of Theorem \ref{thm:ot}}
Consider the two component Legendrian $\Lambda \cup \Lambda^+ \subset (S^3,\alpha_{ot})$ consisting of a knot $\Lambda$ and its push-off $\Lambda^+$ by the positive Reeb flow, constructed inside the complement of an overtwisted disc in $Y$. The main result \cite[Theorem 1.1]{Dimitroglou:Exact} provides an exact Lagrangian torus $\Sigma \subset \R \times S^3$ of vanishing Maslov class which, when intersected with $[-1,1] \times S^3 \subset \R \times S^3$, consists of the pair of Lagrangian cylinders
$$ \Sigma \cap \left([-1,1] \times S^3\right) = [-1,1] \times \left(\Lambda \cup \Lambda^+\right).$$ 
Moreover, the torus $\Sigma$ can be assumed to bound an embedded handle-body.

The goal is to replace the cylinder $[-1,1] \times \Lambda$ by a Lagrangian cylinder smoothly isotopic to $([-1,1] \times \Lambda)_{\iota,A}$ for a map $\iota$ that parametrises a annulus in $[-1,1] \times \Lambda$ whose inclusion is null-homotopic, and where $A$ is an arbitrary knotted arc. Furthermore, we want to perform this construction and the isotopy inside $[-1,1] \times S^3 \setminus \left([-1,1] \times \Lambda^+\right).$ The result then follows from Corollary \ref{cor:tori} above, since there are infinitely many different knot groups of knots in $S^3$.

The construction of the Lagrangian concordance in $[-1,1] \times S^3 \setminus \left([-1,1] \times \Lambda^+\right)$ is performed as in the proof of Theorem \ref{thm:knotted}. Since $\Lambda \cup \Lambda^+$ lives in the complement of an overtwisted disc, we can find a Darboux chart of $S^3$ that is disjoint from $\Lambda^+$, and which intersects $\Lambda$ in a standard Legendrian arc $\Lambda_k \subset J^1\R$ that is obtained from the zero-section $j^10$ by adding $k$ positive and $k$ negative stabilisations, where $k > 0$ is arbitrary. Corollary \ref{cor:lagrel} produces the sought Lagrangian realisation of $([-1,1] \times \Lambda_k)_{\iota,A}$ for $k \gg 0$ sufficiently large.
\qed

\section{Constructing the approximation (Proof of Theorem \ref{thm:main})}

\subsection{A standard neighbourhood theorem}

First we need to establish a suitable standard neighbourhood theorem for two-dimensional totally real submanifolds of four-dimensional symplectic manifolds. Of course, since the symplectic form restricted to a totally real surface does not have a unique form, the standard neighbourhood depends on the totally real embedding.
\begin{prp}
\label{prp:stdform}
A totally real embedding of an open orientable surface 
$$\varphi \colon \Sigma^2 \hookrightarrow (X^4,d\eta,J)$$
for an almost complex structure $J$ compatible with the symplectic form $d\eta$ can be extended to a symplectic embedding
$$\Phi \colon (O,d\lambda) \hookrightarrow (X^4,d\eta)$$ 
of a neighbourhood $O \subset T^*\Sigma$ of the graph $\Gamma_{\varphi^*\eta}$ of the one-form $\varphi^*\eta \in \Omega^1(\Sigma)$, where $\lambda \in \Omega^1(T^*\Sigma)$ denotes the tautological one-form.

In addition, if $\varphi^*\eta$ vanishes in some neighbourhood $U \subset \Sigma$, where we already have a symplectic identification $\Psi \colon (D^*U,d\lambda) \hookrightarrow (X,d\eta)$ on which $\Psi|_{0_U}=\varphi$, then we may assume that $\Phi=\Psi$ holds in $O \cap T^*V \subset O \cap T^*\Sigma$ for any choice of open subset $V \subset U$ with $\overline{V} \subset U$.
\end{prp}
\begin{proof}
The restriction of the tautological one-form $\lambda \in \Omega^1(T^*\Sigma)$ to the graph
$$\Gamma \coloneqq \Gamma_{\varphi^*\eta} \subset T^*\Sigma$$
is equal to $\varphi^*\eta$ itself. Hence, the tautological symplectic form $d\lambda|_{T\Gamma}$ restricted to the same graph is equal to $\varphi^*d\eta$.

The next step is to extend the tangent map
$$ T\varphi \colon T\Gamma \to T\varphi(\Sigma) \subset TX $$
induced by $\varphi$ to a symplectic bundle map
$$ F \colon T(T^*\Sigma)|_\Gamma \to TX|_{\varphi(\Sigma)}.$$
This will be carried out in the remainder of the proof.

Along $\Gamma \subset T^*\Sigma$ we choose a complement to $T\Gamma$ that consists of the Lagrangian planes that are tangent to the fibre-direction of $T^*\Sigma$. Next, we proceed to construct a Lagrangian complement to $T\varphi(\Sigma)$. 

First, choose the smooth family of complements $JT\varphi(\Sigma)$ to $T\varphi(\Sigma)$. The assumption that $\varphi(\Sigma)$ is totally real is equivalent to the fact that $JT\varphi(\Sigma) \cap T\varphi(\Sigma)=\{0\}$ is a transverse intersection. Note that, by the compatibility of $J$, it follows that the application of $J$ restricts to an isomorphism of the two-dimensional vector spaces $(T\varphi(\Sigma),\omega|_{T\varphi(\Sigma)})$ and $(JT\varphi(\Sigma),\omega|_{JT\varphi(\Sigma)})$.

There is a homotopy of $JT\varphi(\Sigma)$ through planes that are transverse to $T\varphi(\Sigma)$ that makes the latter plane field Lagrangian by Lemma \ref{lma:linalg1}. 

Once we have found a Lagrangian complement to $T\varphi(\Sigma)$, it is simply a matter of elementary linear algebra to construct the sought extension $F$ of $T\varphi$ to a symplectic bundle map along $\Gamma$. This we formulate as Lemma \ref{lma:linalg2}.

We can finally use the symplectic neighbourhood theorem \cite[Theorem 3.4.10]{McDuff:Introduction} to deform the infinitesimal symplectomorphism $F$ defined along the graph to a symplectomorphism $\Phi$ defined in a neighbourhood of the graph. Under the additional assumption of the existence of a symplectomorphism $\Psi$ with the required properties defined on $D^*U \subset T^*\Sigma$, the symplectomorphism $\Phi$ can moreover be taken to restrict to $\Psi$ in the specified neighbourhood $O \cap T^*V \subset T^*\Sigma.$
\end{proof}

The following two basic facts about linear non-degenerate two-forms on $\R^4$ with respect to subspaces were used in the proof of the above proposition.

\begin{lma}
\label{lma:linalg1}
Let $V,W \subset (\R^4,\omega_0)$ be a pair of transverse two-planes that satisfy the property that the linear map
\begin{gather*}
\eta \colon V \to W^*,\\
v \mapsto \omega_0(v,\cdot)|_{W},
\end{gather*}
is an isomorphism. (This is e.g.~the case when $V$ is totally real and $W=JV$ for a tame almost complex structure $J$.) Given any basis $\langle \mathbf{f}_1,\mathbf{f}_2 \rangle$ of $W$ the linear span
$$W_{\mathbf{f}_1,\mathbf{f}_2} \coloneqq \left\langle \mathbf{f}_1 - \eta^{-1}\left(\omega_0\left(\mathbf{f}_1,\cdot\right)|_W\right),\mathbf{f}_2 \right\rangle
$$
is a Lagrangian two-plane that is transverse to $V$, and which depends continuously on $V$ together with the choice of ordered basis on $W$.
\end{lma}
\begin{proof}

Note that two two-planes in $\R^4$ are transverse precisely when their intersection is zero-dimensional. For the Lagrangian property, it suffices to compute
\begin{eqnarray*}
\lefteqn{\omega_0(\mathbf{f}_1-\eta^{-1}\left(\omega_0\left(\mathbf{f}_1,\cdot\right)|_W\right),\mathbf{f}_2)=}\\
& = \omega_0(\mathbf{f}_1,\mathbf{f}_2)-\omega_0(\eta^{-1}\left(\omega_0\left(\mathbf{f}_1,\cdot\right)|_W\right),\mathbf{f}_2)=\omega_0(\mathbf{f}_1,\mathbf{f}_2)-\omega_0(\mathbf{f}_1,\mathbf{f}_2)=0.
\end{eqnarray*}
Here we have used the fact that $\omega_0(\eta^{-1}(w^*),\cdot)=w^*(\cdot)$ is satisfied on $W$ for any functional $w^* \in W^*$ by construction, applied to the special case $w^*(\cdot)=\omega_0(\mathbf{f}_1,\cdot)|_W$.
\end{proof}

\begin{lma}
\label{lma:linalg2}
Let $V,W \subset (\R^4,\omega_0)$ be two transverse two-planes, where $W$ is moreover is Lagrangian. For any basis $\langle \mathbf{e}_1,\mathbf{e}_2 \rangle=V$ there exists a basis $\langle \mathbf{f}_1,\mathbf{f}_2 \rangle=W$ determined uniquely by
$$ \omega_0(\mathbf{e}_i,\mathbf{f}_j)=\delta_{i,j},$$
where this basis moreover depends continuously on the choice of basis $\mathbf{e}_1,\mathbf{e}_2$, as well as the planes $V$ and $W$ subject to the above conditions.

In particular, if $V_i,W_i \subset (\R^4,\omega_0)$, $i=0,1,$ are two pairs of transverse two-planes for which $W_i$ are Lagrangian, any linear isomorphism $\phi \colon (V_0,\omega_0|_{V_0}) \xrightarrow{\cong} (V_1,\omega_0|_{V_1})$ that preserves the restriction of the symplectic forms extends to a linear symplectomorphism $\Phi$ of $(\R^4,\omega_0)$ uniquely determined by the requirement $\Phi(W_0)=W_1$. Here $\Phi$ depends smoothly on the initial choice of planes and the identification $\phi$.
\end{lma}
\begin{proof}
The linear map $\eta \colon V \to W^*$ given by $v \mapsto \omega_0(v,\cdot)|_W$ is an isomorphism by the non-degeneracy of $\omega_0$ and the assumption that $W$ is a Lagrangian plane. Any basis $\langle \mathbf{e}_1,\mathbf{e}_2\rangle$ of $V$ thus induces a basis $\langle \eta(\mathbf{e}_1),\eta(\mathbf{e}_2)\rangle$ of $W^*$. The sought basis of $W$ is the corresponding uniquely defined dual basis.

Any symplectomorphism $\Phi$ extending $\phi$ and satisfying the assumption $\Phi(W_0)=W_1$ is uniquely determined as follows. The basis $\langle \mathbf{f}_1,\mathbf{f}_2\rangle$ of $W_0$ that is uniquely determined by the basis $\langle \mathbf{e}_1,\mathbf{e}_2\rangle=V_0$ via $\omega_0(\mathbf{e}_i,\mathbf{f}_i)=\delta_{i,j}$ must be mapped under $\Phi$ to the corresponding basis $\langle \Phi(\mathbf{f}_1),\Phi(\mathbf{f}_2)\rangle=W_1$ that is uniquely determined by the basis $\langle \phi(\mathbf{e}_1),\phi(\mathbf{e}_2) \rangle=V_1$ via $\omega_0(\Phi(\mathbf{f}_i),\phi(\mathbf{e}_i))=\delta_{i,j}$.

To see that there exists such an extension $\Phi$ in the first place, we argue as follows. Start by choosing any basis of $V_0$, which by the above induces a basis of $V_1$ via $\phi$, and then subsequently induces bases of $W_i$ for $i=0,1,$ as above. We then require $\Phi$ to send the induced basis of $W_0$ to the induced basis of $W_1$. Since $W_i$ are Lagrangian, this extension is clearly is a symplectomorphism of $(\R^4,\omega_0)$.
\end{proof}

\subsection{Preparatory modifications of the concordance}

Here we assume that
$$\Sigma \subset ([T_-,T_+ ]_t \times Y,d(e^t\alpha),J)$$
is a totally real concordance between two transverse knots $K_\pm \subset \{T_\pm\} \times Y$, for a compatible almost complex structure $J$ which is cylindrical near the boundary $\{T_\pm\} \times Y$. We further assume that $\partial_t$ is tangent to the concordance $\Sigma$ near the boundary.

We begin by showing how to approximate a knot that is nowhere tangent to the Reeb vector field by a Legendrian representative in the same smooth isotopy class by adding sufficiently many stabilisations. This is a result that was previously proven by Etnyre in \cite[Theorem 2.5]{Etnyre:Legendrian}. In addition we need to control how the framing is changed in this process. A systematic study of Legendrian approximation while staying transverse to a fixed vector field was developed by Cahn and Chernov in \cite{Cahn}. We here give a self-contained account of the approximation result.
\begin{prp}[\cite{Etnyre:Legendrian},\cite{Cahn}]
\label{prp:nowherereeb}
Any smooth knot $K \subset (Y^3,\alpha)$ that is nowhere tangent to the Reeb vector field can be $C^0$-perturbed to a knot $K'$ which is Legendrian in some small arc, where $K$ and $K'$ moreover are isotopic through knots that are nowhere tangent to the Reeb vector field. Furthermore, the entire knot $K$ is smoothly isotopic to a Legendrian knot $\Lambda$ where
\begin{enumerate}
\item the Legendrian $\Lambda$ is smoothly isotopic to the result of $k_+ \gg 0$ positive and $k_- \gg 0$ negative stabilisations of the Legendrian arc in $K'$, by an isotopy through knots that are nowhere tangent to the Reeb vector field;
\item if both $k_+$ and $k_- \gg 0$ are taken sufficiently large, the entire isotopy from $K$ to $\Lambda$ may be assumed to be through knots that all are of arbitrarily small $C^0$-distance from $K$ for suitable choices of parametrisations.
\end{enumerate}
\end{prp}
\begin{rmk}
\label{rmk:tb}

If we endow an arbitrary knot $K$ that is nowhere tangent to the Reeb vector field with the framing induced by the same vector field, then there is typically no framed isotopy to a Legendrian knot endowed with the Reeb framing. The reason is that the self-linking number of Legendrian knots defined using the Reeb framing satisfies the so-called Thurston--Bennequin inequality, as proven by Eliashberg \cite{Contact3Man}. This puts strong restrictions on the allowed Reeb-framings of a Legendrian knot.
\end{rmk}
\begin{proof}
We start by constructing a standard neighbourhood of $\varphi(K)$. Let $\varphi \colon S^1_\theta \hookrightarrow Y$ be a parametrisation of $K$. Writing $f(\theta)d\theta =\varphi^*\alpha$ we consider the section
$$ \Gamma \coloneqq \left\{z=0,p=-f(\theta)\right\} \subset J^1S^1.$$
Since this section has a contact-form preserving embedding into $(Y,\alpha)$ with image $K$, we can use the Reeb flow to extend the embedding of this graph to a contact-form preserving embedding of the annulus
$$ A \coloneqq \left\{z \in [-\epsilon,\epsilon],p=-f(\theta)\right\} \subset \left(J^1S^1,dz-p\,d\theta\right).$$
into $(Y,\alpha)$ which sends $\Gamma \subset A$ to $K$. To that end, recall that the Reeb vector field in $J^1S^1$ is given by $\partial_z$ and, thus, it is tangent to the above annulus $A$.

Then we can further extend this map to a contactomorphism of a neighbourhood
$$ \left\{ |z| \le \epsilon, |p-f(\theta)|\le \epsilon \right\} \subset \left(J^1S^1,dz-p\,d\theta\right) $$
into $(Y,\alpha)$ which is a strict along the annulus $A$. This construction is standard, and can be carried out as the proof of \cite[Theorem 4.1]{Dimitroglou:C0limits}. Namely, one extends the embedding to a smooth embedding of the neighbourhood that is a contactomorphism along the annulus $A$. Then one uses Moser's trick to deform this smooth embedding to a contactomorphism in a neighbourhood of the annulus.

For $\epsilon>0$ sufficiently small, any smooth isotopy of knots whose tangents are of some uniform non-zero distance away from the subspace spanned by $\partial_z$ may be assumed to be nowhere tangent to the Reeb vector field of $(Y,\alpha)$. Indeed, along the annulus $A$ the Reeb vector field of $\alpha$ is identified with $\partial_z$.

The initial isotopy from $K$ to $K'$ which makes $K$ Legendrian in an arc can be carried out by a $C^0$-small deformation of the graph $\Gamma$ of the function $f \colon S^1 \to \R$ in the above coordinates to a graph of a function that is constant near some $\theta_0 \in S^1$.

The construction of the Legendrian $\Lambda$ is then carried out as follows. We can $C^0$-approximate $\Gamma$ by a Legendrian $\Lambda_k \subset J^1S^1$ that is obtained from $j^10$ by adding $k_+$ positive and $k_-$ negative stabilisations, where $k_\pm \gg 0$. Namely, we can consider a knot with front projection in $S^1 \times [-\epsilon,\epsilon]_z$ where the $p$-coordinate is controlled by adding sufficiently many small zig-zags of suitable slopes to the front; see \cite[Figure 7]{Etnyre:Legendrian} as well as Figure \ref{fig:coil} below. Also c.f.~the construction in Section \ref{sec:constructing}.

Note that we need to use positive stabilisations in the region of the knot where $p>0$ and negative stabilisations in the region where $p<0$. We can always add additional small stabilisations of a suitable sign in order to ensure that the number of negative and positive stabilisations become some arbitrarily large numbers $k_\pm \gg 0$.

The existence of a smooth isotopy between the original knot and its Legendrian approximation is obvious, since the stabilisations correspond to Reidemeister-I moves with respect to the Lagrangian projection $J^1S^1 \to T^*S^1$ as knot projection. The difference in winding number can be computed explicitly in the Lagrangian projection of the knot, for which the black-board framing is induced by the Reeb vector field; see Figure \ref{fig:winding}. Both (1) and (2) are now obvious by construction.
\end{proof}

\begin{figure}[htp]
 \vspace{3mm}
 \labellist
 \pinlabel $z$ at 4 94
 \pinlabel $x$ at 99 19
 \pinlabel $x$ at 99 57
 \pinlabel $y$ at 4 42
	\endlabellist
 \includegraphics[scale=2]{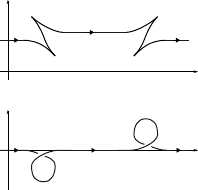}
 \caption{Top: the front projection of a negative and positive stabilisation, shown on the left and right, respectively. Bottom: the corresponding Lagrangian projection. The stabilisation of either sign contributes with the self-linking number $-1$ for the framing given by the Reeb vector field (i.e.~the black board framing with respect to the Lagrangian projection), as can be computed e.g.~using the lower knot diagram.}
 \label{fig:winding}
\end{figure}

\begin{lma}
Let $\Sigma$ be a totally real concordance as above. After extending the cylindrical part by adjoining a trivial cylinder $[T_+,T_+'] \times K_+$ with $T_+' \gg 0$ sufficiently large, we can produce a new totally real concordance $\Sigma' \subset [T_-,T_+']\times Y$ that is Lagrangian near its boundary, which is related to $\Sigma$ by a smooth isotopy of $([T_-,T'_+] \times Y,\{T_-,T'_+\} \times Y)$.
\end{lma}
\begin{proof}

Since the concordance $\Sigma \subset [T_-,T_+] \times Y$ is totally real and tangent to $\partial_t$ near its boundary $\partial \Sigma=K_+ \cup K_-$, the knots $K_\pm \subset Y$ are nowhere tangent to the Reeb vector field.

By Proposition \ref{prp:nowherereeb}, we may find smooth $C^0$-perturbations $K'_\pm$ of $K_\pm$ that are Legendrian in some non-empty sub-arcs. One can readily extend this perturbation to a smooth isotopy that takes $\Sigma$ to $\Sigma'$ through totally real concordances, where the latter thus is a concordance from $K'_-$ to $K'_+$.

Further, Proposition \ref{prp:nowherereeb} implies that there exist Legendrian knots $\Lambda_\pm \subset (Y,\alpha)$ that are isotopic to the result $K''_\pm$ of $k_+$ positive and $k_-$ negative stabilisations of the Legendrian arcs in $K'_\pm$, where the isotopy moreover can be taken through knots that are nowhere tangent to the Reeb vector-field. (Note that we use the same number of negative and positive stabilisations on both knots $K'_\pm$.)

We claim that there is a smooth isotopy of $([T_-,T_+] \times Y,\{T_-,T_+\} \times Y)$ that takes $\Sigma'$ to a totally real concordance with boundary $K''_\pm$ tangent to $\partial_t$ near the boundary. To see this, we can consider an embedded path $(\gamma,\partial \gamma) \subset (\Sigma',\partial \Sigma')$ whose two boundary components are contained in either Legendrian subset of $K'_\pm$. We then perturb $\Sigma'$ by a smooth isotopy through totally real concordances relative boundary in order to make it Lagrangian in some small neighbourhood of the entire path $\gamma$; note that $\Sigma'$ already is Lagrangian near $\partial \gamma$, since the cylinder over a Legendrian arc is Lagrangian. To construct the smooth isotopy, we use Proposition \ref{prp:section} to represent $\Sigma'$ by a section $\Gamma_{\varphi^*\eta} \subset T^*\Sigma$. This section can be $C^0$-approximated by the section of a one-form that is closed in an sufficiently small neighbourhood of $\gamma$, and which coincides with the original section near $\partial \gamma$.

Once $\Sigma'$ has been made Lagrangian near $\gamma$, we can use the symplectic standard neighbourhood theorem to symplectically identify a neighbourhood of $\gamma$ with the zero section in $T^*(\gamma \times [-\epsilon,\epsilon])$. We then replace the zero-section $0_{\gamma \times [-\epsilon,\epsilon]}$ with the cylinder over an unknotted Legendrian arc with $k_+$ positive and $k_-$ negative stabilisations. (Here we need to use the symplectomorphism $\Psi$ described below in order to import the Lagrangian cylinder in the symplectisation into the cotangent bundle.) 
 
Since there is a smooth isotopy $K^t_\pm$ from $K^0_\pm=K''_\pm$ to $K^1_\pm=\Lambda_\pm$ of knots that are nowhere tangent to the Reeb vector field, we can construct the sought concordance by concatenating $\Sigma'$ with the trace of the isotopies $$\left\{(t,x); \:\: x \in K^{\lambda^{-1}t}_\pm \right\} \subset [0,\lambda^{-1}] \times Y$$
at both boundary components of $\Sigma'$. These traces are totally real concordances if we take $\lambda >0$ to be sufficiently small. In addition, they are clearly isotopic to the trivial trace cobordism through totally real concordances.

Finally, note that we must take $T'_+-T_+ \ge 2\lambda^{-1}\gg 0$ in order to have sufficient space to add these traces, which is the reason why it is necessary to extend the concordance by a trivial cylinder.

\end{proof}

By the above it is possible in the following to restrict attention to the case when the totally real concordance $\Sigma \subset ([T_-,T_+]\times Y,d(e^t\alpha))$ is Lagrangian near its boundary. In this case the standard form of the neighbourhood provided by Proposition \ref{prp:stdform} takes the following more specific form. First, we need to make use of the symplectomorphism
\begin{gather*}
\Psi \colon (T^*(\R \times S^1),-d(p_tdt+p_\theta d\theta)) \to (\R_t \times J^1S^1,d(e^t(dz-p\, d\theta)))\\
(t,\theta,p,z)=(t,\theta,e^{-t} p_\theta,e^{-t} p_t)=\Psi(t,\theta,p_t,p_\theta)
\end{gather*}
under which the Liouville vector-field $\partial_t$ in the symplectisation (i.e.~the codomain) corresponds to $\partial_t+p_t\partial_{p_t}+p_\theta\partial_{p_\theta}$ in the cotangent bundle (i.e.~the domain).

\begin{prp}
\label{prp:section}
Assume that $\Sigma \subset ([T_-,T_+] \times Y,d(e^t\alpha))$ is a totally real concordance which is Lagrangian near its boundary. There exists an open symplectic embedding
$$\Phi \colon (U,d(e^t\alpha)) \hookrightarrow (T^*([T_-,T_+] \times S^1),-d(p_t\,dt+p_\theta\,d\theta))$$
where $U$ is a neighbourhood of $\Sigma \subset ([T_-,T_+]\times Y,d(e^t\alpha))$ such that $\Phi(\Sigma)=\Gamma \subset T^*([T_-,T_+] \times S^1)$ is a section which coincides with the zero section near the boundary.

Furthermore, we may assume that the map $\Psi \circ \Phi$ is cylindrical near the boundary, in the sense that it preserves the Liouville flow there. In particular, it follows that
\begin{gather*}
\Psi \circ \Phi \colon U \hookrightarrow \R \times J^1S^1,\\
(t,x) \mapsto (t,\phi_\pm(x))
\end{gather*}
holds near the boundary, where $\phi_\pm$ thus are contactomorphisms.
\end{prp}
\begin{proof}
Note that $\Sigma$ is a cylinder over a Legendrian near its boundary. The Legendrian standard neighbourhood theorem \cite{Geiges:Intro} implies that we can find a standard neighbourhood of the form $[-\epsilon +T_+,T_+] \times J^1K_+$ and $[T_-,T_-+\epsilon] \times J^1K_-$ that identifies the collar of $\Sigma$ near the boundary component $K_\pm$ with the cylinder over $j^10 \subset J^1K_\pm$. Using the inverse $\Psi^{-1}$ of the symplectomorphism defined above we obtain a Weinstein neighbourhood of the above Lagrangian cylinder.

Finally, we apply Proposition \ref{prp:stdform} with the aforementioned identification of the Weinstein neighbourhood of the Lagrangian collar of the boundary of $\Sigma$ defined by $\Psi^{-1}$ near $\partial \Sigma$. This gives the sought symplectomorphism.
\end{proof}

Since the identification $\Phi$ is cylindrical near the boundary, one can extend the neighbourhood $U$ to a neighbourhood $\hat{U} \supset U$ that contains the cylindrical ends $(-\infty,T_-] \times K_-$ and $[T_+,+\infty) \times K_+$, and such that $\Phi$ extends to a cylindrical map defined over the entire non-compact cylindrical ends.

\begin{lma}
\label{lma:pq=0}
After, first, extending $\Sigma$ by adding a cylindrical end $[T_+,T_+'] \times K_+$ for $T_+' \gg T_+$ sufficiently large and, then, performing a $C^2$-small perturbation of $\Sigma$ through totally real surfaces supported in $(T_+,T_+')\times Y$, we may assume the following. The symplectic embedding $\Phi$ provided by Proposition \ref{prp:section} satisfies
$$ \Phi(\Sigma) \subset \{p_t =0\} \subset T^*([T_-,T_{+}'] \times S^1)$$
in addition to the properties specified in that proposition.
\end{lma}
\begin{proof}
We use the coordinates supplied by Proposition \ref{prp:section} in which $\Sigma$ is identified with a section $\Gamma \subset T^*([T_-,T_+]\times S^1)$, and under which a neighbourhood of $\Sigma$ is identified with a neighbourhood of this section. Moreover, the section coincides with the zero-section near $\{t=T_\pm\}$ by construction.

Consider the function
$$h(t,\theta)=\int_{-\infty}^t p_t(\Gamma(s,\theta))ds.$$
Note that $h(t,\theta)$ vanishes near $t=T_-$ and is constant in $t$ near $t=T_+$. After enlarging the cotangent bundle by adding $T^*([T_+,T'_+]\times S^1)$ we can add a small section of the form $p_{t}=-\chi(t)h(T_+,\theta)$ where $\chi(t) \ge 0$ is supported in $[T_+,T'_+]$ and satisfies $\int_\R \chi(t)dt=1$. After taking $T'_+ \gg 0$ we may further assume that $\chi(t)$ is arbitrarily small. This small section corresponds to an extension of $\Sigma \subset [T_-,T_+]\times Y$ to a concordance inside $[T_-,T'_+]\times Y$ that is a small perturbation of the cylinder over $K_+$ inside the subset $[T_+,T'_+] \times Y$.

The function $h(t,\theta)$ defined as above for the new section defined for all $t \in [T_-,T'_+]$ satisfies the property that $h(t,\theta)=0$ on both ends $
\{t \le T_-
\}$ and $\{t \ge T'_+\}$. The fibre-wise addition of $-dh$ is thus a compactly supported symplectomorphism that takes $\Gamma$ to a section contained inside the subset $\{p_t=0\}$ as sought.
\end{proof}

Any smooth isotopy of Legendrian embeddings $\varphi_t \colon \Lambda \hookrightarrow J^1S^1$ induces an immersed Lagrangian cobordism $\Psi(\mathcal{L}_{\{\varphi_t\}})$ in the symplectisation $\R_t \times J^1S^1$ in the following manner. Consider the Lagrangian projection $\mathcal{L}_{\{\varphi_t\}} \subset T^*(\R \times S^1)$ that corresponds to the Legendrian in $J^1(\R \times S^1)=T^*(\R \times S^1) \times \R_Z$ with front projection 
$$ \R \times \Lambda \ni (t,x) \mapsto (t,\theta(\varphi_t(x)),e^t z(\varphi_t(x))) \subset \R_t \times S^1_\theta \times \R_Z.$$
We then apply the above symplectomorphism $\Psi$ to $\mathcal{L}_{\{\varphi_t\}}$ in order to produce the immersed Lagrangian concordance
$$ \Psi\left(\mathcal{L}_{\{\varphi_t\}}\right) \subset \R \times J^1S^1$$
inside the symplectisation. The property of the cobordism to be cylindrical, i.e.~tangent to $\partial_t$, is simply the condition that $\varphi_t(x)$ is independent of $t$.

We proceed to formulate the criterion for embeddedness and size of the cobordism.

\begin{lma}
\label{lma:lagcob}
The above cobordism $\mathcal{L}_{\{\varphi_t\}}$ is embedded if all families of Reeb chords $c_t$ of $\varphi_t(\Lambda) \subset J^1S^1$ have non-decreasing lengths $\ell(c_t)$, i.e.~when $\frac{d}{dt}\ell(c_t)\ge 0$. Moreover, each intersection
$$\gamma_s=\Psi\left(\mathcal{L}_{\{\varphi_t\}}\right) \cap \{t=s\} = \{s\} \times J^1S^1$$
has projection to $T^*S^1$ that coincides with the Lagrangian projection of $\varphi_t(\Lambda)$, while we have a bound
$$|p_t| \le \sup_{t,x} \left|e^{-t}\frac{d}{dt}e^tz(\varphi_t(x))\right|=\sup_{t,x}\left|z(\varphi_t(x))+\frac{d}{dt}z(\varphi_t(x))\right|$$
of its $p_t$-coordinate.
\end{lma}

\begin{proof}
Consider the image of a slice $\mathcal{L}_{\{\varphi_t\}} \cap \{t=s\}$ under the canonical projection $T^*(\R \times S^1) \to T^*S^1$. The image is clearly equal to the Lagrangian projection of the Legendrian $\varphi_{s}(\Lambda) \subset J^1S^1 \to T^*S^1$.

To recover the entire slice $\mathcal{L}_{\{\varphi_t\}} \cap \{t=s\}$ from the projection to $T^*S^1$, we just need the missing coordinate $p_t$, which can be expressed as
$$p_t=\partial_t(e^tz(\varphi_t(x))).$$
The claimed bound on the $p_t$-coordinate easily follows from this identity.

To see the embeddedness property of $\mathcal{L}_{\{\varphi_t\}}$ we note that any double-point of the projection of the slice $\mathcal{L}_{\{\varphi_t\}} \cap \{t=s\}$ corresponds to a Reeb chord $c_s$ on $\varphi_s(\Lambda)$. Using $c_s^+$ and $c_s^-$ to denote the endpoint and starting point of the Reeb chord, respectively, we deduce that the corresponding points in $\mathcal{L}_{\{\varphi_t\}} \cap \{t=s\}$ have different values of the $p_t$-coordinate whenever
$$ \frac{d}{dt}\left(z\left(\varphi_t\left(c_t^+\right)\right)-z\left(\varphi_t\left(c_t^-\right)\right)\right)=\frac{d}{dt} \ell(c_t) \ge 0.$$
It follows that the entire Lagrangian concordance is embedded when the Reeb chords lengths on $\varphi_t(\Lambda)$ are non-decreasing.
\end{proof}

\subsection{Constructing the Lagrangian concordance}
\label{sec:constructing}

We now assume that the concordance satisfies the properties above, so that Proposition \ref{prp:section} and Lemma \ref{lma:pq=0} can be applied to produce a section $\Gamma \subset T^*([T_-,T_+] \times S^1)$ that coincides with the zero-section near the boundary, is contained inside $\{p_t=0\}$, and which is the image of the concordance $\Sigma \subset [T_-,T_+] \times Y$ under a symplectomorphism from a neighbourhood of the latter.

Consider the image
$$\Psi(\Gamma) \subset (\R_t \times J^1S^1,d(e^t\alpha))$$
inside the symplectisation, where we recall that we use the coordinates
$$(J^1S^1,\alpha)=(\R_z \times T^*S^1,\alpha)=(\R_z \times \R_p \times S^1_\theta,dz-p\,d\theta).$$
Further, consider the Lagrangian projections of the slices
$$ \gamma_s \coloneqq \pi_{T^*S^1}(\Psi(\Gamma) \cap \{t=s\}) \subset T^*S^1, \:\: s \in [T_-,T_+],$$
which, since $\Gamma$ is a section, is a smooth family of graphical curves (i.e.~sections) in $T^*S^1$ that coincide with the zero-section for all $s$ close to $s=T_\pm$.

Note that the pull-backs $\gamma_s^*(p\,d\theta)$ are not always exact, which means that the curves $\gamma_s$ do not admit Legendrian lifts in general. The goal is to deform the path of graphical embeddings $\gamma_s$ to a path of exact immersions $\eta_s$ that have a smooth family of Legendrian lifts $\varphi_s \colon S^1 \to J^1S^1$ such that $\pi_{T^*S^1} \circ \varphi_s =\eta_s$. Moreover, we want the Lagrangian trace cobordism
$$\Psi\left(\mathcal{L}_{\{\varphi_t\}}\right) \subset (\R \times J^1S^1,d(e^t\alpha))$$
produced by Lemma \ref{lma:lagcob} to be embedded, arbitrarily close to $\Psi(\Gamma)$ in $C^0$-norm and, moreover, smoothly isotopic to the latter in some small neighbourhood. In particular, this means that the immersion $\eta_s$ must be a $C^0$-approximation of $\gamma_s$.

\emph{Step (I): Adding loops corresponding to zig-zags at $t=T_-$.} Deform the embedding $\gamma_{T_-,0} \coloneqq \gamma_{T_-}=0_{S^1}$ to an immersion $\gamma_{T_-,1}$ by adding $N \gg 0$ pairs of small multiply covered loops that are adjacent to each other, and rotate $W$ turns in different directions. This addition of loops amounts to adding $W\cdot N$ positive and the same number of negative stabilisations to the Legendrian lift $\tilde{\gamma}_{T_-,0}$. We depict such a pair of multiply covered loops and the corresponding stabilised Legendrian lift in Figure \ref{fig:coil}. More precisely, we add such pairs of loops evenly spaced along the zero-section $\gamma_{T_-,0}$, putting one loop at each point $\theta=i 2\pi/{2N}$ with $i=0,1,\ldots 2N-1$. In addition, we make sure that each of the two loops in every pair winds around a disc of the same area $a>0$, and that the pull-back of $p\,d\theta$ to the immersion still is exact (i.e.~the total signed area bounded by the immersion and the zero section vanishes). This is equivalent to the existence of a Legendrian lift $\tilde{\gamma}_{T_-,1}$.

Consequently, the Legendrian lift $\tilde{\gamma}_{T_-,1} \subset J^1S^1$ of $\gamma_{T_-,1}$ is embedded and equal to a deformation of the zero-section $j^10$ by the addition of $ W\cdot N$ positive and $ W\cdot N$ negative stabilisations. Also, note that there is a smooth isotopy from the original Legendrian zero-section $\tilde{\gamma}_{T_-,0}$ to $\tilde{\gamma}_{T_-,1}$. The reason is that stabilisations correspond to Reidemeister-I moves in the Lagrangian projection; see e.g.~Figure \ref{fig:winding}.

At this point in the construction, it is important to note that for any choices of $N, W \ge 0$, we can always find a sufficiently small $a=c_0>0$ in order for the produced Legendrian lift to be arbitrarily $C^0$-close to the original Legendrian lift (i.e.~the zero-section $j^10 \subset J^1S^1$).

\emph{Step (II): Extending the loops to the entire isotopy.} There is an area-preserving isotopy $\psi_s \in\OP{Symp}(T^*S^1,d(p\, d\theta))$, where $T^*S^1=S^1_\theta \times \R_{p}$, that fixes the foliation $\{\theta=\OP{const}\}$ by cotangent fibres and for which $\psi_s(\gamma_{T_-})=\gamma_s$ for $s \in[T_-,T_+]$. In other words, the isotopy preserves the fibres of the bundle $T^*S^1 \to S^1$ set-wise. Using the same isotopy $\psi_t$ we extend the immersion $\gamma_{T_-,1}$ to a smooth family of immersions $\gamma_{s,1}=\psi_s\left(\gamma_{T_-,1}\right)$. Since $\psi_s$ is not necessarily Hamiltonian, the family $\gamma_{s,1}$ is again not exact in general, and hence these curves do not necessarily admit Legendrian lifts.

The parameter $a>0$ in the above construction can be taken to be a smooth parameter in some interval $[c_0,c_1]$ where $0<c_0<c_1$ are small. We will use the notation
$$ \gamma_{s,1}^{\{a\}^{2N}},\:\: a\in [c_0,c_1], $$
where $\{a\}^{2N}=(a,\ldots,a)$ is just an ordered $2N$-tuple with all entries consisting of the single parameter $a$. The reason for this notation is that in the next step we will create a deformation of this family of curves that depends on $2N$ parameters, where at this initial moment all parameters are set to the same value $a$.

\emph{Step (III): Controlling the size of the individual loops.} We then modify the above construction by allowing the area of each of the $2N$ ($W$-fold covered) loops to be controlled individually, instead of all being equal to the same value $a$. In practice, we perform this construction by a fibre-wise interpolation between the curves in the family $\gamma^{\{a\}^{2N}}_{s,1}$ for different values of $a$, supported in regions of the form $\{\theta \in [\theta_i,\theta_{i+1}]\} \subset T^*S^1$. We will need to give this interpolation some special attention in Step (V) below.

The area bounded by the $i$:th loop will be denoted by $a_i \in [c_0,c_1]$, $i=1,\ldots,2N$, where $0<c_0<c_1$ are fixed and small. Denote the resulting family of curves by $\gamma_{s,1}^\mathbf{a}$, which has a smooth dependence on $\mathbf{a} \in [c_0,c_1]^{2N}$. Furthermore, we can construct a family of isotopies $\psi^{\mathbf{a}}_s$, $s\in[T_-,T_+]$, with smooth dependence on $\mathbf{a} \in [c_0,c_1]^{2N}$, that preserves the fibres of $T^*S^1 \to S^1$ set-wise, and for which
$$ \gamma_{s,1}^{\mathbf{a}}=\psi^{\mathbf{a}}_s\left(\gamma_{T_-,1}^{\{c_0\}^{2N}}\right), \:\: \mathbf{a} \in [c_{0},c_{1}]^{2N}.$$
It is important to note that the family produced satisfies the property that the smooth projection $\pi_{S^1} \circ \gamma_{s,1}^{\mathbf{a}}$ of the immersion to the base $S^1$ does not depend on the parameter $\mathbf{a}$.

\emph{Step (IV): Expanding the loops to make the regular homotopy exact with small primitives.}

For any fixed value of the above parameter $\mathbf{a} \in [c_0,c_1]^{2N}$, the integrals
$$z_0(\tau,s) \coloneqq \int_{[0,\tau]} \left(\gamma_{s,1}^{\mathbf{a}}\right)^*(p \,d\theta),$$
where $\tau \in \R$ is a lift of the angular coordinate on $S^1$, are generally neither small nor $2\pi$-periodic in $\tau$ (i.e.~no Legendrian lift exists). The failure of the periodicity is due to the non-exactness of the immersed curves. However, note that the difference
$$z_0(\tau,s)-z_0(\tau+2\pi,s) = c(s) \in \R $$
is a smooth function depending only on $s$ which, moreover, vanishes near $s=T_\pm$.

The goal is now to find a suitable path $\mathbf{a}(s)$, $s \in [T_-,T_+]$, which satisfies the following properties:
\begin{enumerate}[label=(P.\arabic*), ref=(P.\arabic*)]
\item \label{P.1} The corresponding family $\gamma^{\mathbf{a}(s)}_{s,1} \subset (T^*S^1,d(p\,d\theta))$ of immersions should all be \emph{exact}, so that they admit a family $\tilde{\gamma}^{\mathbf{a}(s)}_{s,1}$ of Legendrian lifts that constitute a Legendrian isotopy.
\item \label{P.2} The $z$-coordinates $z(\tau,s)$ of these Legendrian lifts as well as their derivatives $\partial_s z(\tau,s)$ should all stay sufficiently small in the uniform norm, so that the Lagrangian trace cobordism produced by Lemma \ref{lma:lagcob} from this Legendrian isotopy will have $p_t$-coordinate close to the original concordance $\Psi(\Gamma)$; recall that $\Gamma$ has vanishing $p_t$--coordinate and, hence, $\Psi(\Gamma) \subset \R \times J^1S^1$ has vanishing $z$-coordinate.
\item \label{P.3} We need all components $0<c_0 \le a_i(s)<c_1$ of $\mathbf{a}(s)$ to stay sufficiently small for all $s$. This is in order for the obtained regular homotopy $\gamma^{\mathbf{a}(s)}_{s,1}$ to stay $C^0$-close to the original isotopy $\gamma_s$.
\item \label{P.4} All Reeb chords on the Legendrians isotopy $\tilde{\gamma}^{\mathbf{a}(s)}_{s,1}$ must have lengths that are \emph{non-decreasing} in the parameter $s$. This is in order to ensure that Lemma \ref{lma:lagcob} produces an \emph{embedded} Lagrangian trace cobordism. In particular, this means that $a_i'(s) \ge 0$ necessarily holds for all components of $\mathbf{a}(s)$.
\end{enumerate}
To summarise the above properties: for any $\epsilon,c_1>0$ sufficiently small, we want a family $\mathbf{a}(s) \in (0,c_1]$ with $a_i'(s) \ge 0$, for which the primitives of the pull-backs of $p\,d\theta$ satisfy
\begin{equation}
\label{eq:z}
z(\tau,s)=\int_{[0,\tau]} \left(\gamma_{s,1}^{\mathbf{a}(s)}\right)^*(p \,d\theta) \in [-\epsilon,\epsilon], \:\:z(\tau,s)=z(\tau+2\pi,s), \:\: |z(\tau,s)| +|\partial_s z(\tau,s)| \le 3\epsilon, 
\end{equation}
where $\tau \in \R$ is a real lift of the angular coordinate on $S^1_\theta$.

We now explain how this can be achieved for suitable parameters $N$ and $W$. The main mechanism of the addition of the loops is the following effect on the pull-back of $p\,d\theta$ to the curve: If the $i$:th winding loop is positive (resp. negative), then a primitive of this pull-back must decrease (resp. increase) by an amount $W\cdot a_i(s)>0$ as the loop is traversed in the direction of the orientation . Recall that this (locally defined) primitive is the $z$-coordinate of a (locally defined) Legendrian lift.

Property \ref{P.1} is easy to achieve while still ensuring Properties \ref{P.3} and \ref{P.4}. For any small $c_1>0$, we simply need to take both $N,W \gg 0$ sufficiently large. Namely, when
$\int_{S^1} (\gamma^{\mathbf{a}}_{s,1})^*p\,d\theta=c(s)$ satisfies $c(s)>0$ (resp. $c(s)<0$) we simply need to increase the areas $a_{2i}>0$ of the positively winding loops (resp. $a_{2i+1}>0$ of the negatively winding loops) by the amount $|c(s)|/(NW)>0$ in order to make the immersion exact.

The first part of Property \ref{P.2} (uniform smallness of the $z$-coordinate) can be achieved by adjusting the parameters $a_i(s)>0$ so that the corresponding loops cancel the contribution $z_0(\tau,s)$ of the integral coming from the original path $\gamma^{\mathbf{a}}_{s,1}$. Here we need $N,W \gg 0$ sufficiently large, so that this can be done while satisfying Property \ref{P.3}; i.e.~the parameters $a_i(s)>0$ should stay sufficiently small. Why the above can be achieved under the requirement $a_i'(s) \ge 0$, which is necessary if we want to have a chance to ensure Property \ref{P.4}, is explained in the following remark.

\begin{rmk}
\label{rmk:coil}
The problem of constructing the parameters $\mathbf{a}(s)$ with $a_i'(s) \ge 0$ that satisfy the above properties is equivalent to solving the following conceptually simpler problem. Any smooth family $Z_s \colon S^1 \to \R$ of smooth functions has a family of graphs $\Gamma_{Z_s} \times \R$ which can be $C^0$-approximated by the topological boundary $C_s$ of the subgraph of a smooth family of step functions $F_s \colon S^1 \to \R$ with $2N$ steps at the points $2\pi i/(2N) \in S^1$ if we take $N \gg 0$ sufficiently large. (In other words, $C_s$ consists of the graph of the step function $F_s$ with endpoints of the steps connected by vertical segments.) 
 If one would like a geometric description that is even closer to the original problem that we are considering here, one can instead think about $C_s$ as $2N$ horizontally aligned rods joined by thin and vertically aligned ``spring coils'' at their endpoints as shown in Figure \ref{fig:coil}, where the spring coils correspond to the loops in the Lagrangian projection. Finding the smooth family of step functions $F_s$ where no vertical jump is allowed to \emph{shrink} in size -- i.e.~the spring coils can only expand -- is possible if one proceeds as follows. First, deform $F_s$ in order to ensure that all vertical segments remain arbitrarily small, bounded by $\delta>0$. This can be done by choosing $N \gg 0$ even larger and approximating a big step by a staircase consisting of sufficiently many small and short steps. (The choice of $N$ needs to grow with the size of $\delta^{-1}$.) At this point, all steps in $F_s$ are small, but some step might still shrink in length as $s$ evolves. Instead of shrinking the small step, we can leave it fixed in size, and instead introduce an extra staircase that go either upwards or downwards between the previous steps (in practice, this means increasing $N \gg 0$).
\end{rmk}
The choices of parameters $N$, $W$, and $c_1$ (where the latter gives a bound $0<a_i(s)<c_1$) depend on each other in the following manner.
\begin{itemize}
\item The ``maximum step size'' $\delta \coloneqq W \cdot c_1>0$ is chosen to be arbitrarily small, depending on how well we want to control the size of the $z$-coordinate of $\tilde{\gamma}^{\mathbf{a}(s)}_{s,1}$. The reason why we need $\delta>0$ to be small is because, in view of Property \ref{P.4}, any decrement (resp. increment) of the $z$-coordinate introduced by the increased size of a positively (resp. negatively) oriented loop will remain once it has been introduced. (The loops are not allowed shrink, so the only possibility to cancel the decrement (resp. increment) of the $z$-coordinate is to increase the size of some other negatively (resp. positively) winding loop.) Cf. ~Remark \ref{rmk:coil}.
\item We need to choose $N \gg 0$ sufficiently large depending on the magnitudes of $\delta^{-1}=\frac{1}{W\cdot c_1}$, $\max_{\tau,s}|\partial_\tau z_0(\tau,s)|$, and $\max_{\tau \in [0,2\pi],s}|z_0(\tau,s)|$. (Roughly speaking, if each loop only can give a contribution of at most $\delta$ to a change of value of the $z$-coordinate, then we need to use more loops.)
\item Increasing $N \gg 0$ means that the bound $c_1>a_i(s)>0$ on the parameters must be shrunk even further -- the reason is that we must have room to introduce more loops along the curves, which means that they have to wind around less area. Since $\delta=W \cdot c_1$ has been fixed, the parameter $W>0$ thus also depends on $N$. (Roughly speaking, if the loops are supposed to satisfy a smaller bound $a_i(s)<c_1$ on their area, then we might need to increase $W \gg 0$ if we want the small change of the size of loops to still have sufficient effect on the change of the $z$-coordinate.)
\end{itemize}
We thus conclude that the sought family $\mathbf{a}(s)$ with $c_1>a_i(s)>0$ and $a_i'(s) \ge 0$ can be constructed, for which \eqref{eq:z} is satisfied except for the property $|\partial_{ \tau} z(\tau,s)| \le 2\epsilon$. The latter bound will be taken care of in Step (V) below.

Further, since $a'_i(s) \ge 0$ is satisfied, we can readily ensure that the induced fibre-preserving isotopy $\psi_s^{\mathbf{a}(s)}$ of $T^*S^1 \to S^1$ which induces the family of immersions
$$\gamma_{s,1}^{\mathbf{a}(s)}=\psi_s^{\mathbf{a}(s)}\left(\gamma_{T_-,1}^{\{c_0\}^{2N}}\right)$$
is everywhere area non-decreasing. The important outcome of this is ensuring that the Reeb chords of the Legendrian lifts all will be non-decreasing in the parameter $s$. This will be important in the next step, where we make sure that Property \ref{P.4} is satisfied. 

\emph{Step (V): Ensuring that the trace cobordism $\mathcal{L}$ is an embedded $C^0$-approximation of $\Psi(\Sigma)$.} We can achieve the bound
\begin{equation}
|\partial_s z(\tau,s)| \le 2\epsilon,
\end{equation}
i.e.~the second part of Property \ref{P.2}, after a further deformation of the family $\gamma_{s,1}^{\mathbf{a}(s)}$. Namely, since $\pi_{S^1}\circ \gamma_{s,1}^{\mathbf{a}(s)}$ is independent of the parameter $\mathbf{a}(s)$, we can do a linear interpolation of the $p$-coordinates for the discrete set of curves
$$\gamma_{j/M,1}^{\mathbf{a}(j/M)},\:\:\text{for}\:\: j \in \{0,\ldots,M\},$$
with $M \gg 0$, after which we get a new family of immersions for which the corresponding function $z(\tau,s)$ for $s \in [j/M,(j+1)/M]$ becomes the convex interpolation
$$ (\tau,s) \mapsto z(\tau,s)=\frac{(j+1)/M-s}{1/M}z(\tau,j/M) +\frac{(s-j/M)}{1/M}z(\tau,(j+1)/M)$$ 
for all $\tau$. The smoothing of these convex interpolations gives rise to a Legendrian isotopy $\phi_s \colon S^1 \to J^1S^1$.

Furthermore, since all $a_i'(s) \ge 0$ are non-negative, the Reeb chords on this Legendrian isotopy can all be assumed to have non-decreasing lengths; indeed, any Reeb chord length is given by an area enclosed by the Lagrangian projection in $T^*S^1$, which is non-decreasing in $s$ by construction since the isotopy $\psi_s^{\mathbf{a}(s)}$ is locally area non-decreasing. In other words, Property \ref{P.4} is satisfied. By Lemma \ref{lma:lagcob}, a Legendrian isotopy whose Reeb chords have non-decreasing lengths induces a Lagrangian trace cobordism that is embedded. 

\emph{Step (VI): Finishing the construction.} Since $|z(\tau,s)+\partial_sz(\tau,s)| \le 3\epsilon$ can be assumed to be arbitrarily small, the induced bound on the $p_t$-coordinate of the Lagrangian trace cobordism $\mathcal{L}_{\{\varphi_t\}} \subset T^*(\R \times S^1)$ given by Lemma \ref{lma:lagcob} implies that $\mathcal{L}_{\{\varphi_t\}}$ lives in a small neighbourhood of the original section $\Gamma$. Since the Lagrangian projections $\pi_{T^*S^1} \circ \varphi_t$ all are obtained from the original family of embedded curves $\gamma_s$ by a fixed number $2\cdot N\cdot W$ of Reidemeister-I moves (see Figures \ref{fig:winding} and \ref{fig:coil}), one can readily construct a smooth isotopy from $\mathcal{L}_{\{\varphi_t\}}$ to $\Gamma$ inside the neighbourhood that fixes each hypersurface of the form $\{t=t_0\} \subset T^*(\R \times S^1)$ set-wise. Thus, the image $\Psi(\mathcal{L}_{\{\varphi_t\}})$ is the sought Lagrangian concordance. 
\qed

\begin{figure}[htp]
 \vspace{3mm}
 \labellist
 \pinlabel $z$ at 4 128
 \pinlabel $x$ at 113 15
 \pinlabel $x$ at 113 52
 \pinlabel $y$ at 4 39
	\endlabellist
 \includegraphics[scale=2]{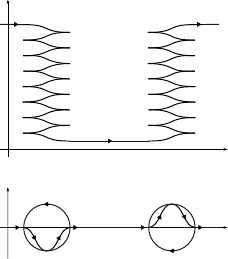}
 \caption{Two canceling coil springs. The zig-zags on the left are positive stabilisations (they increase the rotation number of the Lagrangian projection shown below) while the zig-zags on the right are negative stabilisations (they decrease the rotation number). Here each loop is $W$-fold covered with $W=7$.}
 \label{fig:coil}
\end{figure}

\bibliographystyle{alpha}
\bibliography{references}

\end{document}